\documentclass[10pt,a4paper]{article}
\usepackage[T1]{fontenc}
\usepackage{amsthm}
\usepackage{amssymb}
\usepackage{amsmath}
\usepackage{mathabx}
\usepackage{xcolor}
\usepackage{verbatim, color}
\usepackage{bbm}
\usepackage{float}
\usepackage{dsfont, graphicx}
\usepackage{epstopdf}
\usepackage[colorlinks=true, linkcolor=blue]{hyperref} 
\usepackage[T1]{fontenc}
\usepackage{amsfonts}

\usepackage{thmtools,thm-restate}

\makeatletter
\numberwithin{equation}{section}
\numberwithin{figure}{section}
\theoremstyle{plain}
\newtheorem{thm}{Theorem}[section]
\newtheorem{prop}[thm]{Proposition}

\newtheorem{lem}[thm]{Lemma}

\newtheorem{rem}[thm]{Remark}

  \newcounter{casectr}

\theoremstyle{definition}

\theoremstyle{remark}

\newcommand{\RR}{\mathbb{R}}

\newcommand{\PP}{\mathbb{P}}
\newcommand{\EE}{\mathbb{E}}
\newcommand{\FFF}{\mathcal{F}}

\begin{document}
\title{On blow up NLS with a multiplicative noise}
\author{Chenjie Fan\footnote{State Key Laboratory of Mathematical Sciences, Academy of Mathematics and Systems Science, Chinese Academy of Sciences,
Beijing, China,
fancj@amss.ac.cn}, Junzhe Wang\footnote{School of Mathematical Sciences,
University of Chinese Academy of Sciences,
Beijing, China,
wangjunzhe21@mails.ucas.ac.cn}}
\maketitle

\begin{abstract}
It is of significant interest to understand whether a noise will speed up or prevent blow up. Under certain nondegenerate conditions, \cite{dD2005Blowup} proved a multiplicative noise will speed up blow up of NLS, in the sense that, blow up can happen in any short time with positive probability. We prove that such probability is indeed quite small, and provide a large deviation type upper bound.
\end{abstract}

\section{Introduction}

\subsection{Statement of the main result}

In this note, we consider a focusing stochastic nonlinear Schr{\"o}dinger equation (SNLS) with a  multiplicative noise in $\mathbb{R}^{3}$,
  \begin{equation}\label{eq:SNLS}
    \left\{
    \begin{array}{ll}
    i\partial_t u-\Delta u =|u|^2 u +u\circ \dot{W}, \\
    u(0,x)=u_0(x)\in H^{1}(\mathbb{R}^{3}).
    \end{array}
    \right.
  \end{equation}
 Here, the initial data $u_{0}$ is deterministic. We use $u\circ \dot{W}$  to denote the Stratonovich product. The noise we consider is colored in space and white in time, and we will focus on real-valued noise, but our results extend easily to complex-valued noise. More precisely, we consider\footnote{For simplicity, one may replace the smooth condition on $\phi$ below by $\phi$ be a Hilbert-Schmidt operator from $L^{2}$ to $H^{100}$, and the essence of the article will not change.} $W=\sum_{k=0}^\infty \beta_k(t,\omega)\phi e_k(x)$, with
 \begin{equation*}
  \left\{
    \begin{array}{lr}
    \{e_k(x)\} \text{ be an orthonormal basis of }L^2(\mathbb{R}^3),&\\
    \{\beta_k(t)\}\text{ be i.i.d. Brownian Motions } &\\
    \qquad \,\, \text{in a probability space } (\Omega,\FFF,\PP) \text{ with natural filtration } (\FFF_t)_{t\ge 0}, &\\
    \phi \in R(L^2(\mathbb{R}^3);W^{1,12}(\mathbb{R}^3))\cap R(L^2(\mathbb{R}^3);H^{1}(\mathbb{R}^3)) .&\text{(smooth condition)}
    \end{array}
\right.
 \end{equation*}
Here $R(X;Y)$ means the space of $\gamma$-radonifying operator from Banach space $X$ to $Y$. Our main result is
\begin{thm}\label{thm:main}
  For any initial data $u_0 \in H^1(\RR^3)$, there exists $T_0(\|u_0\|_{H^1})>0$, and a universal constant $\beta > 0$ such that the local solution $u$ of SNLS satisfies
  \begin{equation}
  \ln \PP(u \text{ blows up in } [0,T])\lesssim_{\|u_{0}\|_{H^{1}}, \phi} -T^{-\beta}    
  \end{equation}
  for any $0<T\le T_0$. In particular we can take $\beta=\frac{1}{4}$ here.
\end{thm}

\begin{rem}
  Our universal constant $\beta=\frac{1}{4}$ depends heavily on estimate \eqref{eq:BDG}. It is not expected to be optimal.  See remark \ref{rem:alphabeta} below for possible improvements.  We also note it remains a very interesting problem to obtain a (large deviation type) lower bound.
\end{rem}

\section{Background}
  (Deterministic) NLS is one of the most typical nonlinear dispersive equations, which arises naturally in quantum physics, fluid dynamics, etc. We refer to \cite{tao2006nonlinear},\cite{cazenave2003semilinear} for more background. One may, in general, consider 
  \begin{equation}\label{eq: nlsmain}
    \begin{cases}
    i\partial_t u- \Delta u=\mu |u|^{p-1}u,\\
    u(0,x)=u_{0}
    \end{cases}
  \end{equation}
  When $\mu$ is $+1$, it is called focusing,  when $\mu$ is $-1$, it is called defocusing. It is called mass critical if $p=1+\frac{4}{d} $ and energy critical if $p =1+ \frac{4}{d-2}, d\ge 3$, here we use $d$ to denote dimension. It is well known that $H^{1}$ initial data gives rise to a local solution in the energy critical and energy subcritical setting, \cite{Velo1979Cauchy}, \cite{tao2006nonlinear}, \cite{cazenave2003semilinear}.

  When \eqref{eq: nlsmain} is mass supercritical ($p\geq 1+\frac{4}{d}$) and focusing, one can prove the existence of finite time blow up solutions by Virial arguments, \cite{Glassey1977blowup}.
  Blow-up dynamic is an active research field, for both obstacle arguments and constructive arguments. It is impossible to survey the whole field here, but we would like to mention the recent breakthrough by F. Merle, P. Rapha{\"e}l, I. Rodnianski and J. Szeftel on energy supercritical defocusing case \cite{Merle2022supercritical}, and refer to the reference therein for related works.

  It is natural to extend the study of NLS to the stochastic setting if one believes that the real world is a bit noisy. We refer to \cite{dD1999L2} ,\cite{PhysRevE.49.4627}, \cite{BCRG1995White} for more background. We will focus on the multiplicative noise case.
  
  Local theory has been well understood now for noise which is colored in space and white in time, we refer to \cite{dD1999L2}, \cite{dD2003H1}, \cite{BRZ2014Rescaling}, \cite{BRZ2016stochastic} for subcritical local well-posedness (lwp), and we refer to \cite{FanXu2021critical}, \cite{Zhang2023critical}, \cite{Hornung2018NLS} for critical lwp. We also refer to \cite{FanXuZhao2023Long}, \cite{Fan2024Long}, \cite{Herr2019Scattering} for scattering dynamics. 

  It is a very interesting question to understand how noise can interact with blow up phenomena. There is a folklore in the community that regular noise will accelerate the blow up and space time white noise, we refer to \cite{Millet2020Behavior_colored}, \cite{Millet2021Behavior_white}, \cite{DD2002Numerical_resolution}, \cite{DDD2001theoretical_numerical} for numerical investigations.

  The impact of white noise on blow up is hard to study since even lwp is unclear for the multiplicative noise case. For a regular noise which is colored in space and white in time, A. de Bouard and A. Debussche have generalized the Virial identity to the stochastic setting, \cite{dD2005Blowup}. More importantly, they proved that, in contrast to the determinitic case, solutions to stochastic NLS can blow up in any short time with positive probability. This has been understood as the strong evidence that regular noise accelerated the blow up. We would also like to mention the work \cite{BRZ2017noblowup}, which proves that blow-up can be prevented by non-conservative noises in certain asymptotic sense\footnote{Compared to \cite{dD2005Blowup}, the work \cite{BRZ2017noblowup} considered a family of noise rather than a single noise, and proved that in the limit case, the noise is prevented. }.
  
The current work is motivated by \cite{dD2005Blowup}.  We prove that, for 3d cubic NLS\footnote{It is one of the most typical mass supercrtical and energy subcritical models} with a multiplicative noise, though by the work of \cite{dD2005Blowup}, blow up can happen in arbitrary short time with positive probability under certain non-degenerate condition,  such probability is indeed quite small, and we derive an large deviation upper bound. We remark that our work is of stability argument nature, thus no non-degenerate condition will be needed and one can extend our work easily to other energy subcritical models.
  
We end this subsection by mentioning that there have also been many works on blow-up of stochastic NLS from a constructive perspective, see \cite{SuZhang2023minimalblowup}, \cite{SuZhang2025multi-bubble}, \cite{FSZ2022loglog}.
  
\subsection{More technical background}
  The SNLS equation \eqref{eq:SNLS} can be understood as its equivalent $\mathrm{It\hat{o}}$ form
\begin{equation}\label{eq:SNLS_Ito}
    \left\{
    \begin{array}{ll}
    i\partial_t du-\Delta u =|u|^2 u +u \cdot \dot{W} -\frac{i}{2} u F_\phi, \\
    u(0,x)=u_0(x)\in H^1(\mathbb{R}^3).
    \end{array}
    \right.
\end{equation}
where $F_\phi(x)=\sum_{k=0}^\infty (\phi e_k(x))^2$ is the $\mathrm{It\hat{o}}$-Stratonovich correction and $u \cdot \dot{W}$ is the usual $\mathrm{It\hat{o}}$ product. Its $H^1$ (mild) solution should be understood in the integral formulation 
\begin{equation}\label{eq:SNLS_solution}
  u(t)=S(t)u_0- i\int_0^\tau S(t-s)|u|^2u(s) ds- i\int_0^\tau S(t-s) (u(s) dW_s)- \frac{1}{2}\int_0^\tau S(t-s)(u(s)F_\phi) ds.  
\end{equation}
where $S(t)=e^{-it\Delta}$ is the linear evolution operator of free Schr{\"o}dinger equation and $\tau$ is a stopping time. In some sense, local theory starts from proving $\tau>0$ almost surely. We refer the reader \cite{DaPrato_Zabczyk2014StochInfinite} for further general discussion on stochastic analysis and stochastic PDE.

For stochastic NLS equation \eqref{eq:SNLS}, a local wellposedness theory was developed in \cite{dD2003H1}, \cite{DDD2001theoretical_numerical}, and it shows that there exists a unique maximal stopping time (depending on $u_0$) $\tau^*(u_0,\omega)$ and a unique 
$u(t,x,\omega) \in C_t H_x^1 ([0,\tau) \times \RR^3)$ almost surely such that $(u,\tau^*(u_0))$ is a local $H^1$ solution of SNLS. Moreover, one has the following blow-up alternative
\begin{equation}
  \tau^*(u_0,\omega)=+\infty \text{\, or \,} \lim_{t\rightarrow (\tau^*(u_0,\omega))^-} \|u(t)\|_{H^1(\RR^3)}= +\infty. \text{\, a.s.}
\end{equation}
We will refer $(u,\tau^*(u_0))$ as the $H^1$ solution of SNLS \eqref{eq:SNLS} and $\tau^*$ as its blow-up time.

In this note, we consider real-valued noise, which is called the conservative case because the mass of SNLS $M(u)=\int_{\mathbb{R}^3} |u|^2 dx$ remains conserved here. One may get from \cite{dD2003H1} and this note the same result for complex-valued noise as well. 

\subsection{Notations}\label{section:notations}
  We use $A\lesssim B$ if there is a constant $C$ such that $A\le CB$. We use $A\lesssim_{m} B$ when the constant $C=C_{m}$ depends on some parameter $m$. We say $A\sim B$ if $A\lesssim B$ and $B\lesssim A$. Constant $C$ in the article may change line by line.

  We use the usual Lebesgue spaces $L^{p}$, and Sobolev space $W^{m,p}$, $H^{m}$.
  
  Given two Banach spaces $E,F$, we use $L(E;F)$ to denote the set of bounded linear operators from $E$ to $F$. 
  Given a Hilbert space $H$, and a Banach space $F$, we use $R(H,F)$ to denote the set of $\gamma$-radonifying operator from $H$ to $F$.

\section*{Funding Information and conflict of interest}
This work has been partial supported by NSFC Grant 12471232, 12288201, CAS Project for Young Scientists in Basic Research YSBR-031.

We declare that we have no conflict of interest in this work.

\section*{Acknowledgement}
We thank Prof. Yanqi Qiu for helpful discussion on UMD space. We thank Prof. Yi Huang for pointing out the reference \cite{Pinelis1994Optimum}. 

\section{Preliminaries}
\subsection{Stochastic preliminary}

We refer \cite{BZ1997OnStochastic} for the definition of  $\gamma$-radonifying operator. We simply recall the ideal property of $\gamma$-radonifying operator here. If $H$ is a Hilbert space, $F,B$ are Banach spaces and $T\in L(F,B), K\in R(H,F) $, then $TK \in R(H,B)$ with the following holds true
\begin{equation}\label{eq:ideal}
  \|TK \|_{R(H,B)}\le \|T\|_{L(F,B)} \|K\|_{R(H,F)}.
\end{equation}

We also mention a fact that if $F$ is a Hilbert space as well, then $R(H,F)$ coincides with the set of Hilbert-Schmidt operators from $H$ to $F$.

One key ingredient is the Burkholder-Davis-Gundy(BDG) inequaility with a characterization of the dependence of the constant on $\rho$, which is summarized from \cite{BZ1997OnStochastic} as follows

\begin{lem}[BDG inequaility]\label{lem:BDG}
  If $X$ is an M-type 2 Banach space and $A(s)$ is a $R(H^1 \cap W^{1,12}(\RR^3);X)$-valued adapted square-integrable function, $T > 0$, we have the following estimate for a $X$-valued stochastic integral
  \begin{equation}
  \EE(\sup_{0\le t_0\le T}\|\int_0^{t_0} A(s) dW_s \|^\rho_{X})
  \le  C_{\rho}(X) \EE(\int_0^T \|A\circ \phi\|^2_{R(L^2(\RR^3);X)}ds)^{\frac{\rho}{2}}  
  \end{equation}
  Here $C_{\rho}(X)$ is a constant depending on $\rho, X$, with the estimate
  \begin{equation}\label{eq:Crho}
  C_{\rho}(X)\le (\frac{\rho}{\rho -1} \beta_\rho(X) K_{2,\rho}(X))^{\rho}
  \end{equation}
  where, $\beta_\rho(X)$ is the $\mathrm{UMD}_\rho$ constant of $X$ and $K_{2,\rho}(X)$ is the best constant in Khintchine inequality in space $X$. In particular, for $X=L^{12/5}(\RR^3)$ (indeed any $L^p(\RR^3)$, $p$ fixed ), we have
  \begin{equation}
    C_\rho(X) \lesssim \rho^{3\rho/2}
  \end{equation}
  for $\rho$ large.
\end{lem}
   Estimate \eqref{eq:Crho} follows from \cite{BZ1997OnStochastic} equation (2.25) and the line after (2.33), see also\footnote{There is a slight typo in \cite{BZ1997OnStochastic} Note 1 that it should be $C_p(X)\le \beta_p^p(X) K^p_{2,p}$} Note 1 there. In this note we mainly use BDG inequality with $X=L^{12/5}(\RR^3)$ and $A(s) =S(t-s)(u(s) \, \cdot)$(from \cite{Pisier2016Martingales} Thm 5.22 we know that $L^{12/5}(\RR^3)$ is a type-2 UMD space). It remains to estimate $K_{2,\rho}(X)$ and $\beta_\rho(X)$. One may get from Minkowski inequaility that for $\rho$ large, (we will only need large $\rho$ in our application)
  \begin{equation}
    K_{2,\rho}(X)\le K_{2,\rho}(\RR),
  \end{equation}
  and use the main result of \cite{Haagerup1981Khintchine} 
  \begin{equation}
    K_{2,\rho}(\RR) = 2^{1/2}\pi^{-\frac{1}{2\rho}}(\Gamma(\frac{\rho+1}{2}))^{1/\rho} \lesssim \rho^{1/2},
  \end{equation}
  to get 
  \begin{equation}\label{eq:K_2_rho}
    K_{2,\rho}(X)\lesssim \rho^{1/2}.
  \end{equation}
  From (\cite{Pisier2016Martingales}, Thm 5.13) one knows that 
  \begin{equation}
    \beta_{\rho}(X) \le \alpha(\rho, q) \beta_{q} (X).
  \end{equation}
  with $\alpha= O(p)$ as $p\to \infty$ and $q$ fixed. As a consequence, we get
  \begin{equation}\label{eq:beta_rho}
    \beta_\rho(X)\lesssim \rho
  \end{equation}
  for all $\rho$ large enough. We end this section by noting that we will mainly use the following two estimates, which follows directly from Lemma \ref{lem:BDG}
  \begin{equation}\label{eq:BDG}
  \EE(\sup_{0\le t_0\le T}\|\int_0^{t_0} S(t-s)(u(s) dW_s)\|^\rho_{L^{12/5}})
  \lesssim  \rho^{3\rho/2} \EE(\int_0^T ||S(t-s)(u(s)\phi \, \cdot)||^2_{R(L^2;L^{12/5})}ds)^{\frac{\rho}{2}}
  \end{equation}
  and a direct corollary of it
  \begin{equation}\label{eq:BDG2}
  \EE(\|\int_0^T S(t-s)(u(s) dW_s)\|^\rho_{L^{12/5}})
  \lesssim  \rho^{3\rho/2} \EE(\int_0^T ||S(t-s)(u(s)\phi \, \cdot)||^2_{R(L^2;L^{12/5})}ds)^{\frac{\rho}{2}}.
  \end{equation}

\subsection{Dispersive preliminary}
  We introduce standard dispersive and Strichartz estimates and refer \cite{keel1998endpoint}, \cite{tao2006nonlinear} for details. We say $(q,r)$ is an admissible Strichartz pair in $\RR^3$ if $q,r\in [2,+\infty]$ and they satisfy
  \begin{equation}
    \frac{2}{q}=3(\frac{1}{2}-\frac{1}{r})
  \end{equation}
  We denote conjugate of $p$ as $p'$, which satisfies $\frac{1}{p}+\frac{1}{p'}=1$.
  \begin{lem}[dispersive estimate]\label{lem:dispersive}
    For $p\ge 2$, we have
    \begin{equation}
      \|S(t) f\|_{L^p_x}\lesssim t^{-d(\frac{1}{2}-\frac{1}{p})} \|f\|_{L^{p'}_x}
    \end{equation}
    for any $f(x) \in L^{p'}(\RR^d)$.
  \end{lem}

  \begin{lem}[Strichartz estimate in $\RR^3$]\label{lem:Strichartz}
    For $(q,r)$ an admissible Strichartz pair in $\RR^3$, we have
    \begin{equation}
      \|S(t) f\|_{L_t^q L_x^r(\RR \times \RR^3)} \lesssim \|f \|_{L^2_x(\RR^3)}
    \end{equation}
    for any $f(x) \in L^2_x(\RR^3)$. For any two admissible Strichartz pairs $(q_1,r_1)$,$(q_2,r_2)$ we have
    \begin{equation}
      \|\int_{0}^{t}S(t-s)F(s) ds\|_{L^{q_1}_t L_x^{r_1}(\RR \times \RR^3)} \lesssim \|F\|_{L^{q'_2}_t L_x^{r'_2}(\RR \times \RR^3)}
    \end{equation}
    for any $F(t,x) \in L^{q'_2}_t L_x^{r'_2}(\RR \times \RR^3)$.
  \end{lem}

  We end this section by introducing
  \begin{equation}
    X^1_1([0,t]):=L_t^\infty H_x^1([0,t]\times\RR^3 ); \qquad X^1_2([0,t]):= L_t^8 W_x^{1,12/5}([0,t] \times \RR^3)
  \end{equation}
  and 
  \begin{equation}
    X^1([0,t]):=X^1_1([0,t]) \cap X^1_2([0,t])
  \end{equation}
  in the sense that $\|\cdot \|_{X^1([0,t])}=\|\cdot \|_{X^1_1([0,t])}+\|\cdot \|_{X^1_2([0,t])}$.
\section{Overview}
\subsection{Truncated SNLS Equation}\label{sec:truncated}
An auxiliary truncated equation is useful when considering SPDE that may blow up. Let $\theta: [0,\infty)\rightarrow \RR^+$ be a smooth cut-off function that
\begin{eqnarray}
    \left\{
      \begin{array}{ll}
      \theta(x) = 1 \text{ on } [0,1],\\
      \theta(x) = 0 \text{ on } [2,\infty),\\
      \theta(x) \text{ is monotonically decreasing. }  
      \end{array}
    \right.
\end{eqnarray}
We introduce the following truncated SNLS
  \begin{eqnarray}\label{eq:truncated_SNLS}
    \left\{
    \begin{array}{ll}
    i\partial_t u_R-\Delta u_R =\theta(\frac{\|u_R\|_{X^1([0,t])}}{R}) |u_R|^2 u_R +u_R\circ \dot{W}, \\
    u_R(0,x)=u_0(x).
    \end{array}
    \right.
  \end{eqnarray}
and recall $X^1([0,t])=L_t^\infty H_x^1\cap L_t^8 W_x^{1,12/5}([0,t])$

Equation \eqref{eq:truncated_SNLS} has been studied in \cite{dD2003H1}, and has been shown to be globally wellposed. It plays an important role in the study of local wellposedness. More precisely, for any $u_0\in H^1(\RR^3)$ and any $T>0$, there exists a unique $u_R\in L^8_\omega X^1([0,T])$ such that 
\begin{equation}\label{eq:truncated_SNLS_solution}
  u_R(t)=S(t)u_0- i\int_0^t S(t-s)\theta(\frac{\|u_R\|_{X^1([0,t])}}{R}) |u_R|^2u_R(s) ds- i\int_0^t S(t-s) ( u_R(s) dW_s)- \frac{1}{2}\int_0^t S(t-s)(u_R(s)F_\phi) ds.  
\end{equation}
And if one denotes
\begin{equation}
  \tau_R:=\inf \{t>0 | \|u_R\|_{X^1([0,t])}\ge R \},
\end{equation}
one has that $\tau_R$ is increasing with $R$. If $R'>R$, $u_{R'}$ coincides with $u_R$ on $[0,\tau_R]$. Furthermore, we have blow-up time $\tau^*=\lim_{R\to \infty} \tau_R$. And $u := \lim_{R\to \infty} u_R$ is well-defined for time $\tau < \tau^*$ and coincides with $u_R$ for $\tau \in [0,\tau_R]$. See \cite{dD2003H1} and \cite{dD1999L2} for details.

\subsection{Overview of the proof of Theorem \ref{thm:main}}
Our main theorem \ref{thm:main} follows from
\begin{prop}\label{prop:main}
  There exists a $R(\|u_0\|_{H^1})>0$, $T_0(\|u_0\|_{H^1})>0$, such that 
  \begin{equation}
    \ln \mathbb{P}(\|u_{R}\|_{X^1([0,T])}\ge R)\lesssim -T^{-1/4}
  \end{equation}
  for any $0<T\le T_0$.
\end{prop}

\begin{proof}[Proof of Thm \ref{thm:main} assuming Prop \ref{prop:main}]
  As we mentioned before in section \ref{sec:truncated}, 
  \begin{equation}
    \tau^* =\lim_{R\to \infty} \tau_R.
  \end{equation}
  If $\|u_{R}\|_{X^1([0,T])}< R$ then $\tau_R > T$, and by the monotonicity of $\tau_R$ we have $\tau^* > T$, so no blow-up occurs in $[0,T]$ in that case. In the view above, if $u$ blows up in $[0,T]$ one gets $\|u_{R}\|_{X^1([0,T])}\ge R$ and one has the following estimate
  \begin{equation}
    \ln \mathbb{P}(u \text{ blows up in } [0,T]) \le \ln \mathbb{P}(\|u_{R}\|_{X^1([0,T])}\ge R)\lesssim -T^{-1/4},
  \end{equation}
  which ends our proof with $\beta=\frac{1}{4}$.
\end{proof}

Let $R$ be fixed for the moment (but we will choose its value later). Note that $u_{R}$ is global and has a mild formulation \eqref{eq:truncated_SNLS_solution}. Let
\begin{equation}
  Ju_R:=\int_0^t S(t-s) (u_R(s) dW_s)
\end{equation}
then equation \eqref{eq:truncated_SNLS_solution} can be reduced into
\begin{equation}
  u_R(t)=S(t)u_0- i\int_0^t  S(t-s)\theta(\frac{\|u_R\|_{X^1([0,t])}}{R})|u_R|^2u_R(s) ds- iJu_R - \frac{1}{2}\int_0^t S(t-s)(u_R(s)F_\phi) ds.
\end{equation}
Proposition \ref{prop:main} follows from the following two lemmas
\begin{lem}\label{lem:deterministic}
  Let $R > R_0(\|u_0\|_{H^1})$ be fixed. For all $0< T\le c_1 R^{-4}$ with $c_1$ small enough, if $\|Ju_R\|_{X^1([0,T])} \le \frac{R}{2}$, then $\|u_R\|_{X^1([0,T])}< R$.
\end{lem}

\begin{lem}\label{lem:stochastic}
  Let $R > R_0(\|u_0\|_{H^1})$ be fixed. For all $0< T\le c_2 R^{-4}$ with $c_2$ small enough, we have
  \begin{equation}
  \ln \mathbb{P}(\|Ju_R\|_{X^1([0,T])} > \frac{R}{2})\lesssim - T^{-1/4}
  \end{equation}
\end{lem}

\begin{proof}[Proof of Prop \ref{prop:main} assuming Lemma \ref{lem:deterministic} and Lemma \ref{lem:stochastic}]
  Take $R$ large enough such that the assumption of Lemma \ref{lem:deterministic} and \ref{lem:deterministic} hold and fix this $R$. Let $T_0=c_0 R^{-4}$ here with $c_0=\min(c_1,c_2)$. By Lemma \ref{lem:deterministic}, for all $0<T\le T_0$, if $\|u_R \|_{X^1([0,T])}\ge R$ then one gets $\|Ju_R\|_{X_T} > \frac{R}{2}$ so combine Lemma \ref{lem:stochastic} we have
  \begin{equation}\label{eq:log_Markov}
    \ln \mathbb{P}(\|u_{R}\|_{X^1([0,T])}\ge R)\le \ln \mathbb{P}(\|Ju_R\|_{X^1([0,T])} > \frac{R}{2})\lesssim -T^{-1/4}.
  \end{equation}
\end{proof}

Lemma \ref{lem:deterministic} is purely deterministic, and relies on the stability arguments for NLS while Lemma \ref{lem:stochastic} relies on BDG inequality with the constants $\rho^{3\rho /2}$ (estimate \eqref{eq:BDG}). We will present proofs of these lemmas in the next section.

\begin{rem}\label{rem:alphabeta}
  If one can prove estimate of type \eqref{eq:BDG} with better constant $\rho^{\alpha \rho}$, one may prove Thm \ref{thm:main} with $\beta = \frac{3}{8\alpha}$ with essentialy the same proof in this article.  For example,  one can conclude from (\cite{Pinelis1994Optimum}, Thm 4.3) that $\alpha = \frac{1}{2}, \beta = \frac{3}{4}$. We thank Prof.  Yi Huang for telling us this reference.
\end{rem}

\begin{rem}
  From the proof of Prop \ref{prop:main} and Lemma \ref{lem:stochastic}, one can know that we have $T_0\sim \|u_0\|^{-4}_{H^1}$ in Thm \ref{thm:main}.
\end{rem}

\section{Proof of Lemma \ref{lem:deterministic}, \ref{lem:stochastic}}
 Some technical estimates are useful here. By Sobolev embedding, H\"{o}lder inequality, Lemma \ref{lem:dispersive} and \ref{lem:Strichartz} above we get
  \begin{enumerate}
    \item For all $0<s<t$
    \begin{equation}\label{lem:estimate1}
      \|S(t-s) f\|_{L_x^{12/5}}\lesssim |t-s|^{-1/4} \|f\|_{L_x^{12/7}}.
    \end{equation}
    \item For all $t>0$
    \begin{equation}\label{lem:estimate2}
      \|S(t) f\|_{X^1([0,t])} \lesssim \|f\|_{H^1_x}.
    \end{equation}
    \item For all $t>0$
    \begin{equation}\label{lem:estimate3}
      \||u|^2 u\|_{L_t^{8/7}W_x^{1,12/7}([0,t])} \lesssim \|\|u\|^3_{W_x^{1,12/5}}\|_{L_t^{8/7}([0,t])}\lesssim t^{1/2}\|u\|^3_{X^1_2([0,t])}
    \end{equation}
    \item For all $t>0$
      \begin{equation}
        \|\int_0^t S(t-s)(u(s)F_\phi) ds\|_{X^1([0,t])}\lesssim \|u F_\phi\|_{L_t^{1}H_x^1([0,t])} \lesssim \|\|u\|_{W_x^{1,12/5}} \|F_\phi\|_{W^{1,12}}\|_{L_t^{1}([0,t])}
      \end{equation}
      and note that
      \begin{equation}
        \|F_\phi\|_{W^{1,12}}\le \sum_{k=0}^\infty \|\phi e_k\|^2_{W^{1,12}} \le \|\phi\|^{2}_{R(L^2;W^{1,12})}\lesssim 1
      \end{equation}
      together we have
      \begin{equation}\label{lem:estimate4}
        \|\int_0^t S(t-s)(u(s)F_\phi) ds\|_{X^1([0,t])} \lesssim \|u\|_{L_t^{1}W_x^{1,12/5}([0,t])} \lesssim t^{7/8} \|u\|_{X_2^1([0,t])}
      \end{equation}
  \end{enumerate}

Now we start with Lemma \ref{lem:deterministic}.

\subsection{Proof of Lemma \ref{lem:deterministic}}

take $X^1([0,T])$ norm on both sides of equation \eqref{eq:truncated_SNLS_solution} and almost surely we have
  \begin{eqnarray*}
   &\|u_R(t)\|_{X^1([0,T])} &\le \|S(t)u_0\|_{X^1([0,T])}+\|\int_0^t S(t-s) \theta(\frac{\|u_R\|_{X^1([0,s])}}{R}) |u_R|^2u_R(s) ds\|_{X^1([0,T])}\\
  & &+\|Ju_R\|_{X^1([0,T])}+\frac{1}{2}\|\int_0^t S(t-s)(u_R(s)F_\phi) ds\|_{X^1([0,T])}.
  \end{eqnarray*}
  We claim
  \begin{enumerate}
    \item
    \begin{equation}\label{claim:1-1}
      \|S(t)u_0\|_{X^1([0,T])}\leq C \| u_0\|_{H^1},
    \end{equation}
    \item 
    \begin{equation}\label{claim:1-2}
      \|\int_0^t S(t-s) \theta(\frac{\|u_R\|_{X^1([0,s])}}{R}) |u_R|^2u_R(s) ds\|_{X^1([0,T])}\le C T^{1/2}R^3,
    \end{equation}
    \item 
    \begin{equation}\label{claim:1-3}
      \|Ju_R\|_{X^1([0,T])}\le \|Ju_R\|_{X_T}\leq \frac{R}{2},
    \end{equation}
    \item 
    \begin{equation}\label{claim:1-4}
      \|\int_0^t S(t-s)(u_R(s)F_\phi) ds\|_{X^1([0,T])} \lesssim T^{7/8}\|u_R\|_{X^1([0,T])}.
    \end{equation}
  \end{enumerate}

  \begin{proof}[Proof of Claims 1-4]
    \eqref{claim:1-1} follows from \eqref{lem:estimate2}. \eqref{claim:1-2} follows from Lemma \ref{lem:Strichartz} 
    \begin{equation}
      \begin{array}{lll}
        \|\int_0^t S(t-s) \theta(\frac{\|u_R\|_{X^1([0,s])}}{R}) |u_R|^2u_R(s) ds\|_{X^1([0,T])}&\le& C \|\theta(\frac{\|u_R\|_{X^1([0,s])}}{R}) |u_R|^2u_R(s)\|_{L_t^{8/7}W_x^{1,12/7}}\\
        &\le& C \||u_R|^2u_R(s)\|_{L_t^{8/7}W_x^{1,12/7}},
      \end{array}
    \end{equation}
    and nonlinear estimate \eqref{lem:estimate3} (Note that one has $\|u_R\|_{X_2^1([0,T])}\le 2R$ or otherwise $\theta(\frac{\|u_R\|_{X^1([0,s])}}{R})=0$ here). \eqref{claim:1-3} follows from the assumption of this lemma. \eqref{claim:1-4} follows from \eqref{lem:estimate4}.
  \end{proof}

  Now, let $A_T:=\|u_R\|_{X^1([0,T])}$. To summarize, we have
  \begin{equation}
    A_{T} \le C\| u_0 \|_{H^1} +C T^{1/2}R^3 + \frac{R}{2}+ CT^{7/8}A_{T},
  \end{equation}
  \begin{equation}
    A_{T} \le \frac{C\| u_0 \|_{H^1} +C T^{1/2}R^3 + \frac{R}{2}}{1 - CT^{7/8}}.
  \end{equation}
  Note that $T \le c_1 R^{-4}$, we have
  \begin{equation}
    A_{T} \le \frac{C\frac{\| u_0 \|_{H^1}}{R} +Cc_1^{1/2}+ \frac{1}{2}}{1 - Cc_1^{7/8}R^{7/2}} R .
  \end{equation}
  Finally for all fixed $R>R_0(\| u_0 \|_{H^1})$ sufficiently large such that $C \frac{\| u_0 \|_{H^1}}{R}$ is small enough, one can take $c_1$ small enough to get $A_{T} < R$.

\subsection{Proof of Lemma \ref{lem:stochastic}}
Our first aim is to show that $\|u_R\|_{L^\rho_\omega X^1([0,T])}< R$ if we take $c_2$ sufficiently small and $\rho= (100 T)^{-1/4}$. To do this one needs to take $L^\rho_\omega X^1([0,T])$ norm on both sides of equation \eqref{eq:truncated_SNLS_solution} to get
  \begin{eqnarray*}
   &\|u_R(t)\|_{L^\rho_\omega X^1([0,T])} &\lesssim \|S(t)u_0\|_{L^\rho_\omega X^1([0,T])}+\|\int_0^t S(t-s)\theta(\frac{\|u_R\|_{X^1([0,s])}}{R})|u_R|^2u_R(s) ds\|_{L^\rho_\omega X^1([0,T])}\\
   & &+\|Ju_R\|_{L^\rho_\omega X^1([0,T])}+\|\int_0^t S(t-s)(u_R(s)F_\phi) ds\|_{L^\rho_\omega X^1([0,T])}.
  \end{eqnarray*}

  and we claim the following estimates
  \begin{enumerate}
    \item \begin{equation}\label{claim:2-1}
    \|S(t)u_0\|_{L^\rho_\omega X^1([0,T])}\lesssim \|u_0\|_{H^1},
    \end{equation}
    \item \begin{equation}\label{claim:2-2}
    \|\int_0^t S(t-s) \theta(\frac{\|u_R\|_{X^1([0,s])}}{R}) |u_R|^2u_R(s) ds\|_{L^\rho_\omega X^1([0,T])}\lesssim T^{1/2}R^3,
    \end{equation}
    \item \begin{equation}\label{claim:2-3}
      \|Ju_R\|_{L^\rho_\omega X^1([0,T])}\lesssim \rho^{\frac{3}{2}} T^{3/8} \|u_R\|_{L^\rho_\omega X^1([0,T])},
    \end{equation}
    \item \begin{equation}\label{claim:2-4}
       \|\int_0^t S(t-s)(u_R(s)F_\phi) ds\|_{L^\rho_\omega X^1([0,T])} \lesssim T^{7/8} \|u_R\|_{L^\rho_\omega X^1([0,T])}.
    \end{equation}
  \end{enumerate}

  \begin{proof}[proof of the claim 1,2,4]
    \eqref{claim:2-1},\eqref{claim:2-2}, \eqref{claim:2-4} follows from \eqref{lem:estimate2}, \eqref{claim:1-2}, \eqref{claim:1-4} and take $L^\rho_\omega$ norm on the both side respectively.
  \end{proof}

  \begin{proof}[proof of the claim 3]
    \eqref{claim:2-3} is rather involved and it can be divided into two estimates
    \begin{equation}\label{claim:2-3-1}
          \|Ju_R\|_{L^\rho_\omega L^8_t W_x^{1,12/5}([0,T])}\lesssim \rho^{\frac{3}{2}} T^{3/8} \|u_R\|_{L^\rho_\omega L_t^\infty H^1_x([0,T])}
    \end{equation}
  \begin{equation}\label{claim:2-3-2}
          \|Ju_R\|_{L^\rho_\omega L_t^\infty H^1_x([0,T])}\lesssim \rho^{\frac{3}{2}} T^{3/8} \|u_R\|_{L^\rho_\omega L^8_t W_x^{1,12/5}([0,T])}
    \end{equation}
    Firstly one needs the following type of BDG inequality deriving from \eqref{eq:BDG2}
    \begin{equation}\label{eq:derivative_BDG}
      \mathbb{E}(\|\int_0^t S(t-s)u_R(s)dW_s\|^\rho_{W^{1,12/5}})\lesssim \rho^{3\rho/2}\mathbb{E}(\int_0^t \|S(t-s)(u_R \phi \, \cdot)\|^2_{R(L^2;W^{1,12/5})} ds)^{\rho/2}.
    \end{equation}
    This estimate is true because one has the following derivative estimate
    \begin{equation}\label{eq:derivative_estimate}
      \mathbb{E}(\|\nabla \int_0^t S(t-s)u_R(s)dW_s\|^\rho_{L^{12/5}})= \mathbb{E}(\|\int_0^t S(t-s)\nabla (u_R(s)dW_s)\|^\rho_{L^{12/5}}),
    \end{equation}
    Using \eqref{eq:BDG2}, equation \eqref{eq:derivative_estimate} can be controlled by
    \begin{equation}
      \lesssim \rho^{3\rho/2}\mathbb{E}(\int_0^t \|S(t-s)\nabla(u_R\phi \, \cdot)\|^2_{R(L^2;L^{12/5})} ds)^{\rho/2},
    \end{equation}
    and this term can be controlled by the RHS of \eqref{eq:derivative_BDG} because
    \begin{equation}
        \| S(t-s)(u_R \phi \, \cdot)\|^2_{R(L^2;W^{1,12/5})} = \|S(t-s)(u_R \phi \, \cdot)\|^2_{R(L^2;L^{12/5})} + \|S(t-s)\nabla (u_R \phi \, \cdot)\|^2_{R(L^2;L^{12/5})}.
    \end{equation}
    Next, one needs the following dispersive estimate deriving from \eqref{lem:estimate1}
    \begin{equation}\label{eq:derivative_dispersive}
      \|S(t-s) f\|_{W^{1,12/5}}\lesssim |t-s|^{-1/4} \|f\|_{W^{1,12/7}},
    \end{equation}
    
    Denote linear operator $K:W^{1,12} \cap H^1 \to W^{1,12/5}$, $K= S(t-s)(u_R \, \cdot)$, from \eqref{eq:ideal} one has
    \begin{equation}\label{eq:S_norm}
      \| S(t-s)(u_R \phi \, \cdot)\|^2_{R(L^2;W^{1,12/5})} \le \| K\| \|\phi \|_{R(L^2;W^{1,12} \cap H^1)}\lesssim \|K\|
    \end{equation}
    Thus we need to compute the operator norm of $K$. For all $v \in W^{1,12} \cap H^1$ one has
    \begin{equation}
      \|S(t-s)(u_R v)\|_{W^{1,12/5}}\lesssim |t-s|^{-1/4} \|u_R v\|_{W^{1,12/7}} \lesssim |t-s|^{-1/4} \|u_R \|_{H^1} \|v \|_{W^{1,12}},
    \end{equation}
    where we use \eqref{eq:derivative_dispersive}. This gives
    \begin{equation}\label{eq:K_norm}
      \|K\|\lesssim |t-s|^{-1/4} \|u_R \|_{H^1}.
    \end{equation}
    From \eqref{eq:derivative_BDG}, \eqref{eq:S_norm}, \eqref{eq:K_norm}, one has
    \begin{equation}\label{eq:BDG_t}
        \mathbb{E}(\|Ju_R\|^\rho_{W^{1,12/5}})=\mathbb{E}(\|\int_0^t S(t-s)u_R(s)dW_s\|^\rho_{W^{1,12/5}})
        \lesssim \rho^{3\rho/2} \mathbb{E}(\int_0^t |t-s|^{-1/2} \|u_R\|^2_{H^1} ds)^{\rho/2}
        \lesssim \rho^{3\rho/2} t^{\rho/4} \|u_R\|^{\rho}_{L_\omega^\rho H_x^1} 
    \end{equation}
    and we have
    \begin{equation}
        \|Ju_R\|_{L^\rho_\omega L^8_t W_x^{1,12/5}([0,T])} = (\mathbb{E}(\int_{0}^{T}\|Ju_R\|^8_{W_x^{1,12/5}} )^{\rho/8})^{1/\rho} \lesssim (\int_{0}^{T}\mathbb{E}(\|Ju_R\|^\rho_{W^{1,12/5}_x})^{8/\rho} ds)^{1/8}
    \end{equation}
    where we use the Minkowski inequality and assume $\rho \ge 8$ here. Combine this with \eqref{eq:BDG_t} we have 
    \begin{equation}
      \|Ju_R\|_{L^\rho_\omega L^8_t W_x^{1,12/5}([0,T])} \lesssim (\int_0^T \rho^{12} t^{2} \|u_R\|_{L_\omega^\rho H_x^1}^{8}dt)^{1/8} \lesssim \rho^{3/2} T^{3/8}\|u_R\|_{L_t^\infty L_\omega^\rho H_x^1} \lesssim \rho^{3/2} T^{3/8}\|u_R\|_{L_\omega^\rho L_t^\infty H_x^1},
    \end{equation}
    where we use the Minkowski inequality once again in the last inequality. This ends the proof of \eqref{claim:2-3-1}. For \eqref{claim:2-3-2}, one can argue as above using BDG inequality \eqref{eq:BDG} to get
    \begin{equation}
      \mathbb{E}(\sup_{0\le t\le T}\|Ju_R\|^\rho_{H^1})=\mathbb{E}(\sup_{0\le t\le T} \|\int_0^t S(t-s)u_R(s)dW_s\|^\rho_{H^1}) \lesssim \rho^{3\rho/2} \mathbb{E}(\sup_{0\le t\le T} \int_0^t \|G\|^2 ds)^{\rho/2}
    \end{equation}
    where $G$ is the linear operator $G:W^{1,12} \cap H^1 \to H^1$, $G= S(t-s)(u_R \, \cdot)$. For all $v \in W^{1,12} \cap H^1 $ one has
    \begin{equation}
      \|S(t-s)(u_R v)\|_{H^1}= \|u_R v\|_{H^1} \lesssim \|u_R \|_{W^{1,12/5}} \|v \|_{W^{1,12}}
    \end{equation}
    so $\|G\| \lesssim \|u_R \|_{W^{1,12/5}} $, and we have
    \begin{equation}
      \mathbb{E}(\sup_{0\le t\le T} \|Ju_R\|^\rho_{H^1})\lesssim \rho^{3\rho/2} \mathbb{E}(\sup_{0\le t\le T} \int_0^t \|u_R \|_{W^{1,12/5}}^2 ds)^{\rho/2} =\rho^{3\rho/2} \mathbb{E}(\int_0^T \|u_R \|_{W^{1,12/5}}^2 ds)^{\rho/2}
    \end{equation}
    Use H\"{o}lder inequality to get
    \begin{equation}
      \mathbb{E}(\sup_{0\le t\le T} \|Ju_R\|^\rho_{H^1})\lesssim \rho^{3\rho/2} T^{3\rho/8}\mathbb{E}(\int_0^T \|u_R \|^8_{W^{1,12/5}} ds)^{\rho/8},
    \end{equation} 
    This ends our proof if $\rho=(100 T)^{-1/4} \ge 8$, which can be obtained if we let $c_2$ small enough.
  \end{proof}

  Let $B_T:=\|u_R(t)\|_{L^\rho_\omega X^1([0,T])}$. To summarize, we have
  \begin{equation}
    B_T \le C\|u_0\|_{H^1} + CT^{1/2} R^3+ C\rho^{3/2} T^{3/8} B_T+CT^{7/8} B_T,
  \end{equation}

  \begin{equation}
    B_T\le \frac{C\|u_0\|_{H^1} + CT^{1/2} R^3}{1-C\rho^{3/2} T^{3/8}- CT^{7/8}}
  \end{equation}
  Recall $T\le c_2 R^{-4}$ and we have
  \begin{equation}
    B_T\le \frac{C\frac{\|u_0\|_{H^1}}{R} + C c_2^{1/2}}{1-C\rho^{3/2} c_2^{3/8} R^{-3/2}- C c_2^{7/8} R^{-7/2}} R
  \end{equation}

  Finally for all fixed $R>R_0(\| u_0 \|_{H^1})$ sufficiently large such that $C \frac{\| u_0 \|_{H^1}}{R}$ is small enough, one can take $c_2$ small enough to get $B_{T} < R$, which ends our aim. Now use Markov inequality and \eqref{claim:2-3} to get
  \begin{equation}\label{eq:Markov}
    \mathbb{P}(\|Ju_R\|_{X^1([0,T])} > \frac{R}{2})
    \le \frac{\|Ju_R\|_{L^\rho_\omega X^1([0,T])}}{(R/2)^\rho} 
    \le C\frac{\rho^{\frac{3\rho}{2}}T^{\frac{3\rho}{8}}R^\rho }{(R/2)^\rho}
    =Ce^{\rho \ln 2+\frac{3\rho}{8} \ln(\rho^4 T)}
    = C e^{-C T^{-1/4}}.
  \end{equation}
  Take logarithm of \eqref{eq:Markov} to yield 
  \begin{equation}
    \ln \mathbb{P}(\|Ju_R\|_{X^1([0,T])} > \frac{R}{2})\le \ln C -C T^{-1/4} \lesssim -T^{-1/4}.
  \end{equation}
  Here the last inequality holds if $c_2$ is small enough, which is our desired result \eqref{eq:log_Markov}.
  
\bibliographystyle{siam}
\bibliography{refers}

\end{document}